\documentclass[11pt]{article}

\usepackage{amssymb,amsmath,amsfonts,epsf,amsmath}
\usepackage{enumerate}
\usepackage{array}
\usepackage[english]{babel}

\usepackage{tikz}

\usetikzlibrary{decorations.pathreplacing,automata,calc,positioning}
\usetikzlibrary{arrows}

\newtheorem{thm}{Theorem}[section]
\newtheorem{prop}[thm]{Proposition}
\newtheorem{obs}[thm]{Observation}

\newtheorem{cor}[thm]{Corollary}
\newtheorem{lema}[thm]{Lemma}

\newtheorem{q}{Question}

\newcommand{\efface}[1]{}

\newcommand{\qed}{\hfill $\square$ \medskip}

\def\cp{\,\square\,}

\begin{document}

\title{Bootstrap percolation in strong products of graphs}

\author{
Bo\v stjan Bre\v sar$^{a,b}$, Jaka Hed\v zet$^{b,a}$
}

\date{}

\maketitle

\begin{center}
$^a$ Faculty of Natural Sciences and Mathematics, University of Maribor, Slovenia\\
\medskip

$^b$ Institute of Mathematics, Physics and Mechanics, Ljubljana, Slovenia\\
\medskip

\end{center}


\begin{abstract}
Given a graph $G$ and assuming that some vertices of $G$ are infected, the $r$-neighbor bootstrap percolation rule makes an uninfected vertex $v$ infected if $v$ has at least $r$ infected neighbors. The $r$-percolation number, $m(G,r)$, of $G$ is the minimum cardinality of a set of initially infected vertices in $G$ such that after continuously performing the $r$-neighbor bootstrap percolation rule each vertex of $G$ eventually becomes infected. In this paper, we consider  percolation numbers of strong products of graphs. If $G$ is the strong product $G_1\boxtimes \cdots \boxtimes G_k$ of $k$ connected graphs, we prove that $m(G,r)=r$ as soon as $r\le 2^{k-1}$ and $|V(G)|\ge r$. As a dichotomy, we present a family of strong products of $k$ connected graphs with the $(2^{k-1}+1)$-percolation number arbitrarily large. We refine these results for strong products of graphs in which at least two factors have at least three vertices.  In addition, when all factors $G_i$ have at least three vertices we prove that $m(G_1 \boxtimes \dots \boxtimes G_k,r)\leq 3^{k-1} -k$ for all $r\leq 2^k-1$, and we again get a dichotomy, since there exist families of strong products of $k$ graphs such that their $2^{k}$-percolation numbers are arbitrarily large. While $m(G\boxtimes H,3)=3$ if both $G$ and $H$ have at least three vertices, we also characterize the strong prisms $G\boxtimes K_2$ for which this equality holds. Some of the results naturally extend to infinite graphs, and we briefly consider percolation numbers of strong products of two-way infinite paths.  

\end{abstract}


\medskip
\noindent{\bf AMS Subj. Class.:} 05C35, 05C76, 60K35

\section{Introduction}
Given a graph $G$ and an integer $r \geq 2$, the {\em $r$-neighbor bootstrap percolation} is an update rule for the states of vertices in $G$. At any given time the state of a vertex is either {\em infected} or {\em uninfected}. From an initial set of infected vertices further updates occur simultaneously and in discrete intervals: any uninfected vertex with at least $r$ infected neighbors becomes infected, while infected vertices never change their state. Given a graph $G$, the smallest cardinality of a set of initially infected vertices, which results in all vertices of $G$ being infected after the $r$-neighbor bootstrap percolation process is finished, is the {\em $r$-percolation number}, $m(G,r)$, of $G$. 

The origins of bootstrap percolation come from physics of ferromagentism and go back to 1979~\cite{cha-1979}, while in 1998 the concept was considered in the context of spreading an infection in square grid networks~\cite{bal-1998}. Balogh and Bollob\'{a}s considered the random bootstrap percolation in hypercubes, where the main challenge is to find the tresholds for probabilities of vertices being initially set as infected in order to get all vertices of the graph infected with probability $1$ or $0$, respectively; see also a related study considering square grids~\cite{bal-2012}. Recently, Przykucki and Shelton considered $m(G,r)$ where $G$ is a $d$-dimensional square grid~\cite{prz-2020}, while Bidgoli et al.~\cite{bid-2021} considered bootstrap percolation of specific Hamming graphs, namely the Cartesian powers of complete graphs. The common feature of the above mentioned investigations of bootstrap percolation is that they all involve various types of Cartesian products of graphs. Beside the Cartesian product operation, the $r$-neighbor bootstrap percolation was considered with respect to the minimum degree of a graph~\cite{gun-2020}, and from the complexity point of view concerning the time (i.e., number of percolation steps) that takes to infect the entire graph~\cite{mar-2018}. 

A special attention was given to the case $r=2$. For instance, Dairyko et al.~\cite{dai-2020} presented Ore-type and Chv\'{a}tal-type conditions related to degrees of a graph $G$ that enforce $m(G,2)=2$. Morris in~\cite{mor-2009} provided some bounds on the minimal bootstrap percolation sets in rectangular grids, where a set is minimal if it yields an infection of the whole graph while none of its proper subsets do it. As it turns out, the $2$-neighbor bootstrap percolation coincides with the concept from graphs convexity; notably, for the so-called $P_3$-convexity, as introduced by Centeno et al.~\cite{cen-2010}, the $P_3$-hull number of a graph $G$ is exactly $m(G,2)$. The $P_3$-convexity and the corresponding hull number were studied in comparison with other convexity parameters~\cite{cen-2013,coe-2019}, and were also considered in specific graph classes such as Kneser graphs~\cite{gri-2021} or Hamming graphs~\cite{bre-2020}.   Coelho et al.~\cite{coe-2019} performed a systematic study of the $P_3$-hull number in graph products. While the Cartesian product seems to be the most challenging one for the bootstrap percolation, for the strong product $G\boxtimes H$ of any non-trivial connected graphs $G$ and $H$ they proved that $m(G\boxtimes H,2)=2$. 

In this paper, we widely extend the study from~\cite{coe-2019} by investigating the $r$-percolation numbers in strong products of graphs. 
Strong product is one of the four standard graph products~\cite{produkti}. Its structure and high density provide several applicable properties. Shannon in~\cite{S} introduced a concept that arises in information theory, which is now known as the Shannon capacity, and is defined through the independence number of strong powers of a graph; see also~\cite{L} for the first major breakthrough in its study. It is also worth noting that every connected graph can be isometrically embedded into the strong products of paths~\cite{sch}, while the well-known Helly graphs are precisely the weak retracts of strong products of paths~\cite{NR}. The strong product of (two-way infinite) paths itself is a natural lattice (also known as the King grid), in which numerous invariants and their applications have been considered; see~\cite{CC,DG,FM} for a short selection of recent references. 

Additional motivation for studying bootstrap percolation in strong products of graphs comes from applications. In particular, in distributed computing one wants to construct a network so as to reduce the failure probability~\cite{K}. As we demonstrate in this paper, the density of the ``strong product network'' ensures the smallest possible value of the bootstrap percolation numbers in many strong products of graphs (see Section~1.2), hence it provides a convenient setting for such a network. 
In a different language, our results show that many strong products of graphs are ``$r$-bootstrap good'' in the sense of the recent study~\cite{B}. In some other applications, standard edges  represent short-range links, while long-range links may also be considered~\cite{Gao}; in such situations the Cartesian grid $\mathbb{Z}\cp\mathbb{Z}$ may as well be replaced by the strong grid as soon as all long-range links have range between $\sqrt{2}$ and $2$. 

Another perspective on our study can be obtained by comparing it with known studies on bootstrap percolation in Cartesian products of graphs. Even for hypercubes, which are in a sense the simplest Cartesian products of graphs, the $r$-neighbor percolation numbers have still not been determined when $r\ge 4$ (the exact values for $m(Q_n,2)$ can be found e.g.\ in~\cite{bre-2020} and for $m(Q_n,3)$ in~\cite{MN}). When $r\geq 4$, the asymptotic behavior of upper bounds on $m(Q_n,r)$ has been extensively explored~\cite{bbm,MN}, while good lower bounds seem to be more elusive, and there are considerable gaps between known lower and upper bounds. In light of these difficulties in Cartesian product graphs, obtaining the exact values of the $r$-neighbor bootstrap percolation numbers in (many) strong products of graphs, which is done in Section~\ref{sec:k-factors}, provides an interesting contrast. In many cases, we do this by finding appropriate sets of vertices that propagate, which unlike in the case of hypercubes reaches the lower bound. In Section~\ref{sec:two-factors}, where we deal with small values of $r$, we study structural properties of (strong product) graphs and/or obtain results that are dependent on percolation numbers of factor graphs. In particular, the result on strong prisms involves the graphs $G$ with $m(G,2)\in\{2,3\}$, which may be of help in better understanding of these graphs. Yet another different approach is needed in Section~\ref{sec:infinite} when dealing with powers of infinite paths, where we need to provide an argument showing that the percolation process advances on every step.

\subsection{Formal definitions and notation}

All graphs considered in this paper are simple and connected. Given a graph $G$ a vertex $x\in V(G)$ is a {\em cut-vertex} if $G-x$ is disconnected. The {\em neighborhood}, $N_G(v)$, of a vertex $v\in V(G)$ is the set of all vertices in $G$ adjacent to $v$, and the {\em closed neighborhood} of $v$, is defined as $N_G[v]=N_G(v)\cup\{v\}$. 
Vertices $u$ and $v$ in a graph $G$ are {\em (closed) twins} if $N_G[u]=N_G[v]$. That is, $u$ and $v$ are adjacent and have the same neighborhoods. 

We follow with a formal definition of the {\em $r$-neighbor bootstrap percolation}.
Let $A_0 \subseteq V(G)$ be an initial set of infected vertices, and, for every $t \geq 1$, let $$A_t = A_{t-1} \cup \lbrace v\in V(G):\, |N(v) \cap A_{t-1}|\geq r \rbrace .$$
The set $A_t \setminus A_{t-1}$ is referred to as vertices infected at time $t$. A vertex $v$ is infected before $u$ if $v \in A_t$, for some $t \geq 0$, while $u \notin A_t$ . We say that $A_0$ {\em percolates} (or is a {\em percolating set}) if $\bigcup\limits_{t\ge 0}^{} A_t= V(G)$. 

A natural extremal problem is to find a smallest percolating set $S=A_0$. For any graph $G$ and $r\geq 2$, let $$m(G,r)=\min \Big\{ |A_0|: \, A_0 \subseteq V(G), \: \bigcup\limits_{t=0}^{\infty}A_t = V(G)\Big\}.$$
Any percolating set $S$ satisfying $m(G,r) = |S|$ is thus a {\em minimum percolating set}, and $m(G,r)$ is the {\em $r$-percolation number} of $G$. Clearly, $m(G,r)\ge r$ for all $r\le |V(G)|$.


The {\em strong product} of graphs $G$ and $H$ is the graph $G \boxtimes H$, whose vertex set is $V(G) \times V(H)$, and two vertices $(g,h)$ and $(g',h')$ are adjacent precisely if one of the following is true: 
\begin{enumerate}
\item[•] $g=g'$ and $hh' \in E(H)$, or
\item[•] $h=h'$ and $gg' \in E(G)$, or
\item[•] $gg' \in E(G)$ and $hh' \in E(H)$. 
\end{enumerate}

By $G^h = \lbrace (g,h): \, g \in V(G) \rbrace$ we denote the subset of $V(G \boxtimes H)$ called the {\em $G$-layer} on vertex $h$, and, by abuse of language, $G^h$ also denotes the subgraph of $G\boxtimes H$ induced by the vertices of the $G$-layer on $h$. Clearly, $G^h$ is isomorphic to $G$ for every $h\in V(H)$. Similarly, for $g \in V(G)$, the {\em $H$-layer} on vertex $g$ is $^g\!H = \lbrace (g,h) : \, h \in V(G) \rbrace$. 

Note that strong product operation is associative and commutative, and $G_1\boxtimes\cdots \boxtimes G_k$ has $V(G_1\boxtimes \cdots \boxtimes G_k)=V(G_1)\times \cdots \times V(G_k)$, and $(x_1,\ldots,x_k)(y_1,\ldots,y_k)\in E(G_1\boxtimes\cdots \boxtimes G_k)$ if and only if $x_i=y_i$ or $x_iy_i\in E(G_i)$ for all $i\in [k]$. If for a factor $G_i$ in the strong product $G_1\boxtimes\cdots \boxtimes G_k$ we have $|V(G_i)|=2$, we say that $G_i$ is an {\em edge-factor} of the strong product. Graph $K_1$ is said to be {\em trivial}, and if $G_i$ is a factor of a strong product with $|V(G_i)|=1$, $G_i$ is a {\em trivial factor}. In this paper, we will only consider strong products in which all factors are non-trivial.

\subsection{Main results and organization of the paper}
\label{sec:org}

In this paper, we consider percolation numbers of strong products of non-trivial graphs.
More precisely, we study $m(G_1\boxtimes\cdots\boxtimes G_k,r)$ depending on the number of factors $k$ and the threshold $r$.  In Section~\ref{sec:k-factors}, we study general upper bounds on the percolation numbers of strong products of graphs, which in many cases lead to exact values. In particular, it is often the case that the best possible value, $m(G,r)=r$, is obtained when $G$ is the strong product of $k$ factors and the threshold $r$ is bounded by a function of $k$. The results depend also on the number of factors in the strong product that have at least three vertices, and are illustrated in the following table:

\begin{center}
\begin{tabular}{|l|c|c|}
 \hline
  $r$ & $m(G,r)\le$ & \# non-edge factors \\ 
 \hline
 $\leq \mathbf{2^{k-1}}$ & $r$ & 1 \\ 
 \hline
 $\leq \mathbf{3\cdot2^{k-2}}$ & $r$ & 2 \\ 
 \hline
 $\leq 7\cdot2^{k-3}$ & $7\cdot2^{k-3}$ & 3\\ 
 \hline
 $\leq 2^{k}-1$ & $3^{k-1}-k$ & $k$\\ 
 \hline
\end{tabular}
\end{center}
The table gives the bounds on $m(G_1\boxtimes\cdots\boxtimes G_k,r)$, depending on the number of non-edge factors. Bounds in the first two lines are bold, which  is to indicate the fact that in all these cases $m(G,r)=r$. In particular, the first line is given by Corollary~\ref{cor:r<=2^(k-1)} and states that $m(G_1 \boxtimes \dots \boxtimes G_k,r)=r$ whenever $r \leq 2^{k-1}$ and $k\geq 2$. As a dichotomy we present an example showing that $m(G_1 \boxtimes \dots \boxtimes G_k,r)$ is not only greater than $r$, but can even be arbitrarily large as soon as $r=2^{k-1}+1$. Next, we prove in Theorem~\ref{thm:veliko_faktorjev_vsaj_dva_3} that $m(G,r)=r$ if $r \leq 3\cdot2^{k-2}$ and $G$ is the strong product of $k$ factors at least two of which are not $K_2$, and a similar dichotomy is proved also in this case. If there are at least three non-edge factors, we can further increase the threshold as shown in the third line (see~Theorem~\ref{thm:veliko_faktorjev_vsaj_trije_3}), while the last line presents an upper bound when all factors have at least three vertices  (see~Theorem~\ref{thm:veliko_faktorjev_vsi_3}).

To see that the bound $r\le 2^k-1$ in the last line of the above table is best possible, consider $G = G_1 \boxtimes \cdots \boxtimes G_k$, where $k \geq 2$ and $G_i$ are connected graphs such that $\delta(G_i) = 1$ for every $i \in [k]$. From the definition of the strong product it follows that $\delta(G)=2^k-1$. Therefore, whenever $r \geq 2^k$, every vertex of degree $\delta(G)$ must be included in the set of initially infected vertices. For instance, if $G_i$ is isomorphic to the star $K_{1,n}$ for every $i \in [k]$, then $G$ has $n^k$ vertices of degree $2^k-1$ and therefore $m(G,r) \geq n^k$ for every $r \geq 2^k$. Noting that $n \in \mathbb{N}$ can be arbitrarily large, we derive the following

\begin{obs}
\label{obs1} If $r \geq 2^k$, then for every integer $M$ there exist graphs $G_1,\ldots,G_k$ such that $m(G_1\boxtimes \cdots\boxtimes G_k,r)>M$. 
\end{obs}

In Section~\ref{sec:two-factors}, we consider percolation numbers of strong products of graphs with only two factors. When $r=3$ and both $G$ and $H$ have order at least $3$, Theorem~\ref{thm:veliko_faktorjev_vsaj_dva_3} implies that $m(G \boxtimes H,3)=3$.  Thus we  consider the only remaining case for $m(G\boxtimes H,3)$, which is when one of the factors is $K_2$, and we prove a characterization of the graphs $G$ such that $m(G \boxtimes K_2,3)=3$. Furthermore, if $G$ and $H$ have the property that $m(G,2)=2$ and $m(H,2)=2$, then $m(G\boxtimes H,4)\le 5$, and if both $G$ and $H$ are not $K_2$, then $m(G\boxtimes H,5)$ can also be bounded from above (see Theorem~\ref{thm: implikacija za m(G krat H, 4)}).

In Section~\ref{sec:infinite}, we consider a natural extension of percolation to infinite graphs. Note that the original definition works also in the case $G$ is an infinite graph, where the only (silent) modification is that the initial set of infected vertices may need to be infinite in order to percolate, in which case we set the $r$-percolation number to be infinite. In this vein, Theorem~\ref{thm:veliko_faktorjev_vsi_3} can also be applied to strong products of infinite graphs. In particular, we infer that $m(\mathbb{Z}^{\boxtimes,n},2^{n}-1)\le 3^{n-1}-n$, where $\mathbb{Z}^{\boxtimes,n}$ is the strong product of $n$ two-way infinite paths $\mathbb{Z}$. It is natural to consider the {\em finiteness percolation threshold} of a graph $G$, which is the supremum of the set of thresholds $r$ for which $m(G,r)<\infty$; we denote this number by ${\rm fpt}(G)$. It is easy to prove that $2^{n}-1\le {\rm fpt}(\mathbb{Z}^{\boxtimes,n})\le 3^{n-1}$, and we establish that for $n\in \{2,3\}$ the upper bound is actually the exact value. In Section~\ref{sec:conclude}, we pose some open problems that arise from this study.

\section{Strong products of graphs with $k$ factors}
\label{sec:k-factors}

In this section, we consider upper bounds and exact results for the percolation number of the strong product $G_1\boxtimes\cdots\boxtimes G_k$, where for the threshold $r$ we have $r<2^k$. (As mentioned in Section~\ref{sec:org}, there is no general upper bound for $m(G_1\boxtimes\cdots\boxtimes G_k,r)$ when $r\ge 2^k$.) The results are divided into several subsections depending on the number of non-edge factors (that is, the number of factors with at least three vertices). 
The case when all factors of the strong product of graphs are $K_2$ is trivial, hence we first consider the most general case when there is at least one non-edge factor. 

\subsection{At least one non-edge factor}

We start by considering the $r$-neighbor bootstrap percolation in strong products of $k$ graphs, where the threshold $r$ is at most by $2^{k-1}$.  Clearly, $m(G,r)\ge r$ for any graph $G$ with  $|V(G)|\ge r\ge 2$. As we will see in the next result(s), if $r\le 2^{k-1}$, then $m(G,r)=r$ where $G$ is a strong product of graphs with $k$ factors. 

\begin{thm} \label{thm:produkt_k_faktorjev}
If $2\le k\leq r \leq 2^{k-1}$ and $G$ is the strong product $G_1 \boxtimes \dots \boxtimes G_k$, where $G_i$ are connected graphs so that $|V(G)| \geq r$, then $m(G,r)=r$. 
\end{thm}
\begin{proof}
Let $k\ge 2$ and $r\in \{k,\ldots, 2^{k-1}\}$ (note that $k\geq 2$ implies $k\le 2^{k-1}$).
Consider the strong product $G_1\boxtimes \dots \boxtimes G_k$, where factors are connected, and the order of the product is at least $r$. 
For all $i \in [k]$, let $|V(G_i)| = n_i$. Since $G_i$ is connected, it contains a BFS-tree. Let us denote the vertices of $G_i$ by $v_1^i,v_2^i,\dots , v_{n_i}^i$ such that $v_1^i$ is the root of the BFS tree, and for each $j$, $2 \leq j \leq n_i$, let $p(v_j^i)=v_{\ell}^i$, be the parent of $v_j^i$, where $\ell < j$. In particular, the parent (and a neighbor) of $v_2^i$ is $v_1^i$, while $v_3^i$ has $v_2^i$ or $v_1^i$ as the parent (and a neighbor).  

Since $m(G,r) \geq r$ is clear, it suffices to find a set $S\subset V(G)$ of size $r$ that percolates. Let $S$ be any subset of the set $\lbrace (v_i^1,v_i^2,\dots , v_i^k): \, i \in \lbrace 1,2 \rbrace \rbrace$, such that $|S| = r$. Such a set $S$ always exists because $r \leq  2^{k-1} < 2^k$.

First note that every vertex $x = (x^1,\dots, x^k)$, where $x^i \in \lbrace v_1^i, v_2^i \rbrace$ for all $ i \in [k]$, gets infected, since $x$ is in $S$ or is a neighbor of all vertices in $S$. We will use induction to prove that eventually every vertex of $G$ gets infected. Let $t_i\in [n_i]$ for all $i\in [k]$.  We claim that every vertex $x=(x^1,\dots, x^k)\in V(G)$, where $x^i \in \lbrace v_j^i: \, j\le t_{i}\}$ for all $i \in [k]$, gets infected. The induction is on $\sum_{i=1}^k{t_i}$ where for the base case we can take $\sum_{i=1}^k{t_i}=2k$; that is $t_i=2$ for all $i\in [k]$.  Thus the base of induction is that all vertices $x=(x^1,\dots, x^k)$, whose coordinates, $x^i$, are in $\{ v_j^i:\, j\le 2\}$ are infected, which has already been proved.

In the inductive step we assume that every vertex $x=(x^1,\dots, x^k)\in V(G)$, where $x^i \in \lbrace v_j^i: \, j\le t_{i}\}$ for all $i \in [k]$ and some $t_i \leq n_i$, is infected. 
In addition, we may assume there exists an index $s\in [k]$ such that $t_s<n_s$, and without loss of generality, let $s=1$. Consider a vertex $x = (x^1,\dots, x^k)$, where $x^1 = v_{t_1+1}^1$ and $x^j \in \lbrace v_1^j, v_2^j, \dots ,v_{t_j}^j \rbrace$ for all $j \neq 1$. Note that $x$ is adjacent to vertices $(p(v_{t_1+1}^1),y^2,\dots,y^k)$ where $y^i \in \lbrace x^i, p(x^i) \rbrace$ for $2 \leq i \leq k$. Since these vertices are all infected, $x$ has at least $2^{k-1}\ge r$ infected neighbors, therefore it gets infected. We have thus proved that all the vertices $x=(v_\ell^1,x^2,\dots, x^k)\in V(G)$, where $\ell\in [t_1+1]$ and $x^i \in \lbrace v_j^i: \, j\in [t_{i}]\}$ for all $i >1$, get infected, which concludes the proof of the inductive step. 
\qed
\end{proof}

\medskip

When considering the graph $G=G_1\boxtimes \cdots \boxtimes G_k$ in the $r$-neighbor bootstrap percolation when $r<k$, we can write $G=G_1 \boxtimes \dots \boxtimes G_{r-1} \boxtimes (G_r \boxtimes \dots \boxtimes G_k)$, and consider $G_r \boxtimes \dots \boxtimes G_k$ as a sole factor. Hence, applying Theorem \ref{thm:produkt_k_faktorjev}, we infer the following 

\begin{cor}
\label{cor:r<=2^(k-1)}
If $k \geq 2$ and $G = G_1 \boxtimes \dots \boxtimes G_k$ for non-trivial connected graphs $G_i$, then $m(G,r)=r$ for all $r \leq 2^{k-1}$. 
\end{cor}

By letting $k=2$ in Theorem~\ref{thm:produkt_k_faktorjev}, and noting that the $2$-neighbor bootstrap percolation coincides with $P_3$-hull convexity, we get the result of Coelho~\cite[Theorem 3.1]{coe-2019} regarding the $P_3$-hull number of the strong product of two graphs. 

Since the case $r\le 2^{k-1}$ is completely resolved, we continue by investigating strong products of $k$ graphs and the $r$-neighbor bootstrap percolation, where $r>2^{k-1}$. The following results yields a dichotomy to the corollary above by showing that as soon as $r>2^{k-1}$, the $r$-percolation number of the strong product of $k$ factors can be arbitrarily large.

\begin{thm} \label{thm:k_faktorjev_rvelik}
If $n\ge 3$, then $m(C_n \boxtimes K_2 \boxtimes \dots \boxtimes K_2, 2^{k-1} +1) = 2^{k-1}-1 + \lceil \frac{n}{2} \rceil$, where $K_2$ appears as a factor $(k-1)$-times.
\end{thm}
\begin{proof}
Note that $K_2 \boxtimes \dots \boxtimes K_2$ is isomorphic to the complete graph $K_{2^{k-1}}$ and let $G = C_n \boxtimes K_{2^{k-1}}$. 
Denote $V(C_n)=\lbrace v_1, \dots , v_n \rbrace$ and $V(K_{2^{k-1}})=\lbrace 1,2,\dots ,2^{k-1} \rbrace$. We start by proving the lower bound $m(G, 2^{k-1} +1)\ge 2^{k-1}-1 + \lceil \frac{n}{2} \rceil$.

Let $S$ be a minimum percolating set of $G$. 
For all $i \in [n]$, let $H_i=G[\lbrace (v_i,p),(v_{i+1},p) : \, p \in [2^{k-1}] \rbrace]$ be the subgraph of $G$, where $i$ is taken with respect to modulo $n$. Clearly, $H_i$ is the union of two $H$-layers $^{v_i}\!K_{2^{k-1}}$ and $^{v_{i+1}}\!\!K_{2^{k-1}}$ and is isomorphic to the complete graph on $2^k$ vertices. In addition, every vertex in $H_i$ has exactly $2^{k-1}$ neighbors in $G-V(H_i)$. Since $r=2^{k-1}+1$, we derive that 
\begin{equation}
\label{eq:SinH_i}
|S\cap V(H_i)|\ge 1, \textrm{ for all } i\in [n]\,.
\end{equation}
Without loss of generality, renaming the vertices of $G$ if necessary, we may assume that $(v_2,1)\notin S$ is a vertex infected at step $1$ of the percolation process. Hence, there are at least $2^{k-1}+1$ initially infected vertices within $^{v_1}\!K_{2^{k-1}}\cup \, ^{v_2}\!K_{2^{k-1}}\cup \, ^{v_3}\!K_{2^{k-1}}=H_1\cup H_2$. When $n\in\{ 3,4\}$, we have $|S|\ge  2^{k-1}-1 + \lceil \frac{n}{2} \rceil$, 
which proves the desired lower bound. Let $n\ge 5$. Note that $|\{H_4,\ldots, H_{n-1}\}|\ge 1$, and using \eqref{eq:SinH_i}, we get $$|S|=|S\cap (H_1\cup H_2)|+|S\cap (H_4\cup\cdots\cup H_{n-1})|\ge 2^{k-1}+1+\left\lfloor \frac{n-3}{2} \right\rfloor.$$
If $n$ is even, then $\lfloor \frac{n-3}{2} \rfloor = \frac{n}{2}-2$, and if $n$ is odd, then $\lfloor \frac{n-3}{2}\rfloor = \frac{n+1}{2} -2$. In both cases, we get $m(G,2^{k-1}+1) \geq 2^{k-1} -1 + \lceil \frac{n}{2} \rceil$, as desired. 

To prove the upper bound, $m(G, 2^{k-1} +1)\le 2^{k-1}-1 + \lceil \frac{n}{2} \rceil$, we consider two possibilities for a percolating set $S$ with respect to the parity of $n$. If $n$ is even, let $$S = \,^{v_1}\!K_{2^{k-1}} \cup \lbrace (v_3,1),(v_5,1),\dots ,(v_{n-1},1) \rbrace,$$ while if $n$ is odd, let $$S =  \,^{v_1}\!K_{2^{k-1}}  \cup \lbrace (v_3,1),(v_5,1),\dots ,(v_n,1) \rbrace.$$ In either case, $|S| = 2^{k-1}-1 + \lceil \frac{n}{2} \rceil$. It is easy to see that $S$ percolates, and so the proof is complete. \qed
\end{proof}

\subsection{At least two non-edge factors}

Theorem~\ref{thm:k_faktorjev_rvelik} shows that if the strong product has only one non-edge factor, Theorem~\ref{thm:produkt_k_faktorjev} is best possible. If there are at least two non-edge factors, we can improve Theorem~\ref{thm:produkt_k_faktorjev} as follows.

\begin{thm}
\label{thm:veliko_faktorjev_vsaj_dva_3}
Let $G$ be the strong product $G_1 \boxtimes \dots \boxtimes G_k$ of connected graphs $G_i$, $i\in [k]$. If $|V(G)|\ge r$ and at least two of the factors have order at least $3$, then $m(G,r)=r$ for all $2 \leq r \leq 3\cdot 2^{k-2}.$ 
\end{thm}

\begin{proof}
The result for $r \leq 2^{k-1}$ follows from Theorem \ref{thm:produkt_k_faktorjev}. 
We start with the proof of the statement, when $r=3\cdot 2^{k-2}$. (The cases when $r\in \{2^{k-1}+1,\ldots, 3\cdot 2^{k-2}-1\}$ will be dealt with in the final paragraph of this proof.)

Let $n_1\ge n_2 \geq 3$, and $n_i \geq 2$ for all $i\in \{3,\ldots, k\}$. Denote $V(G_i) = \lbrace v_1^{(i)} , v_2^{(i)},\dots ,v_{n_i}^{(i)} \rbrace$ for all $i \in [k]$. Since $G_1$ is connected of order at least $3$, it contains a path on three vertices (not necessarily induced). Assume without loss of generality, renaming vertices if necessary, that
$P:v_1^{(1)}v_2^{(1)}v_3^{(1)}$ is a path in $G_1$.  For each $i\in \{3,\ldots, k\}$, let $F_i=\{v_1^{(i)}, v_2^{(i)}\}$ consist of two adjacent vertices. Let 
\begin{equation}
\label{eq:1}
S=V(P)\times\{v_1^{(2)}\}\times \prod_{i=3}^{k}{F_i}.
\end{equation}
Clearly, $|S|=3\cdot 2^{k-2}$. We claim that $S$ percolates. 

Firstly, assuming that vertices in $S$ are infected, we show that vertices in $V(P)\times V(G_2) \times \prod_{i=3}^{k}{F_i}$ become infected. Since $G_2$ is connected, it suffices to show that the set $V(P)\times\{v_s^{(2)}\}\times \prod_{i=3}^{k}{F_i}$ being infected implies that the set $V(P)\times\{v_t^{(2)}\}\times \prod_{i=3}^{k}{F_i}$ becomes infected, where $v_t^{(2)}$ is adjacent to $v_s^{(2)}$ in $G_2$. Indeed, each vertex $(v_2^{(1)},v_{t}^{(2)},v_{i_3}^{(3)},\dots ,v_{i_k}^{(k)})$, where $i_p \in [2]$ for all $p\in \{3,\ldots,k\}$, is adjacent to all vertices in $V(P)\times\{v_s^{(2)}\}\times \prod_{i=3}^{k}{F_i}$. Since $|V(P)\times\{v_s^{(2)}\}\times \prod_{i=3}^{k}{F_i}|=3 \cdot 2^{k-2}$, we infer that vertices $(v_2^{(1)},v_{t}^{(2)},v_{i_3}^{(3)},\dots ,v_{i_k}^{(k)})$, where $i_p \in [2]$ for all $p\in \{3,\ldots,k\}$, become infected. 
Now, consider any vertex $(v_1^{(1)},v_{t}^{(2)},v_{i_3}^{(3)},\dots , v_{i_k}^{(k)})$, where $i_p \in [2]$ for all $p\in \{3,\ldots,k\}$, and note that it is adjacent to all vertices $(v_{i_1}^{(1)},v_{s}^{(2)},v_{i_3}^{(3)},\dots ,v_{i_k}^{(k)})$, where $i_p \in [2]$ for all $p\in [k]\setminus\{2\}$, as well as to all (newly infected) vertices $(v_2^{(1)},v_{t}^{(2)},v_{i_3}^{(3)},\dots ,v_{i_k}^{(k)})$, where $i_p \in [2]$ for all $p\in\{3,\ldots, k\}$. Thus, altogether,  $(v_1^{(1)},v_{t}^{(2)},v_{i_3}^{(3)},\dots , v_{i_k}^{(k)})$  is adjacent to $2^{k-1} + 2^{k-2} = 3\cdot 2^{k-2}$ infected neighbors, and so it gets infected. 
By symmetry, we infer the same fact about vertices $(v_3^{(1)},v_{t}^{(2)}, v_{i_3}^{(3)},\dots ,v_{i_k}^{(k)})$, where $i_p \in [2]$ for all $p\in\{3,\ldots, k\}$. Hence,  vertices in $V(P)\times V(G_2) \times \prod_{i=3}^{k}{F_i}$ become infected, as claimed.  

Secondly, we prove that all vertices in $V(P)\times\prod_{i=2}^{k}{V(G_i)}$ become infected. For this purpose, we claim that for any $i\in \{2,\ldots,k-1\}$, vertices in $V(P)\times\prod_{j=2}^{i+1}{V(G_j)}\times \prod_{j={i+2}}^{k}{F_j}$ get infected assuming that vertices in $V(P)\times\prod_{j=2}^{i}{V(G_j)}\times \prod_{j={i+1}}^{k}{F_j}$ are infected. (Since vertices in $V(P)\times V(G_2) \times \prod_{i=3}^{k}{F_i}$ became infected as proved in the previous paragraph, the truth of this claim implies the statement that vertices in $V(P)\times\prod_{i=2}^{k}{V(G_i)}$ become infected.)
By using the assumption, note that all vertices in $V(P)\times\prod_{j=2}^{i}{V(G_j)}\times \{v_1^{(i+1)}\} \times \prod_{j={i+2}}^{k}{F_j}$ are infected. Now, consider arbitrary adjacent pairs of vertices $v_{j_1}^{(j)}$ and $v_{j_2}^{(j)}$ in $G_j$, where $j\in \{2,\dots,i\}$, and let $W_j=\{v_{j_1}^{(j)},v_{j_2}^{(j)}\}$. Since $G_{i+1}$ is connected, it suffices to show that the set $V(P)\times\prod_{j=2}^{i}{W_j}\times \{v_s^{(i+1)}\} \times \prod_{j={i+2}}^{k}{F_j}$ being infected implies that the set $V(P)\times\prod_{j=2}^{i}{W_j}\times \{v_t^{(i+1)}\} \times \prod_{j={i+2}}^{k}{F_j}$, where $v_t^{(i+1)}$ is adjacent to $v_s^{(i+1)}$ in $G_{i+1}$, becomes infected. To see this, one can use analogous arguments as in the proof in the previous paragraph. Since pairs of vertices in $W_j$ were chosen arbitrarily, we infer that vertices in $V(P)\times\prod_{i=2}^{k}{V(G_i)}$ become infected, as claimed. 

Thirdly, let $P':v_{i_1}^{(2)}v_{i_2}^{(2)}v_{i_3}^{(2)}$ be an arbitrary path in $G_2$, and note that any vertex of $G_2$ lies on such a path, since $G_2$ is connected and $|V(G_2)|\ge 3$. Further, let $W_j=\{v_{j_1}^{(j)},v_{j_2}^{(j)}\}$ consist of arbitrary adjacent vertices in $G_j$, for all $j\in \{3,\dots,k\}$. Note that all vertices in $\{v_1^{(1)}\}\times V(P')\times\prod_{j=3}^{k}{W_j}$ are already infected. Again, by using symmetric arguments as earlier, one can prove that vertices in $\{v_t^{(1)}\}\times V(P')\times\prod_{j=3}^{k}{W_j}$ become infected assuming that vertices in $\{v_s^{(1)}\}\times V(P')\times\prod_{j=3}^{k}{W_j}$ are infected, where $v_s^{(1)}v_t^{(1)}\in E(G_1)$. Since, $G_1$ is connected, we deduce that $V(G_1)\times V(P')\times\prod_{j=3}^{k}{W_j}$ gets infected. Noting that vertices in $P'$ and $W_j$, where $j\in \{3,\dots,k\}$, were arbitrarily chosen, we get that $V(G)$ becomes infected, and $S$ is indeed a percolating set. 

Finally, let $r\in \{2^{k-1}+1,\ldots, 3\cdot 2^{k-2}-1\}$. Note that it suffices to find a set $S'$ in $G$ of size $r$ such that $S$, as defined in~\eqref{eq:1}, becomes infected assuming that $S'$ is infected. Now, let $S'=S\setminus U$, where $U$ consists of any $3\cdot 2^{k-2}-r$ vertices in $\{v_2^{(1)}\} \times \{v_1^{(2)}\}\times\prod_{i=3}^k {F_i}$. (Since $r\in \{2^{k-1}+1,\ldots, 3\cdot 2^{k-2}-1\}$, we have $|U|\le 3\cdot 2^{k-2}-(2^{k-1}+1)<2^{k-2}$, hence $U$ is well defined.) Since $S$ is isomorphic to $P_3\boxtimes K_{2^{k-2}}$, vertices in $\{v_2^{(1)}\} \times \{v_1^{(2)}\}\times\prod_{i=3}^k {F_i}$ are adjacent to all other vertices in $S$. Hence, if $S'$ is initially infected, $S$ becomes infected, and so $S'$ is a percolating set. 
\qed 
\end{proof}

\medskip

We have shown that $r\le 3\cdot 2^{k-2}$ implies $m(G,r)=r$, whenever $G$ is a strong product of $k$ graphs at least two of which have order at least $3$. Now, the natural question is whether $m(G,r)$ is bounded from above also if $r>3\cdot 2^{k-2}$, and the next results shows this is not always true. On the contrary, $m(G,3\cdot 2^{k-2} +1)$ can be arbitrarily large.

\begin{thm}
If $n\ge 3$, then $m(C_n \boxtimes C_n \boxtimes K_2 \boxtimes \dots \boxtimes K_2, 3\cdot 2^{k-2} +1) \geq \lceil \frac{n}{2} \rceil$, where $K_2$ appears as a factor $(k-2)$ times.
\end{thm}

\begin{proof}
Note that $K_2 \boxtimes \dots \boxtimes K_2$ is isomorphic to the complete graph $K_{2^{k-2}}$ and let $G = C_n \boxtimes C_n \boxtimes K_{2^{k-2}}$. Denote $V(C_n)=\{ v_1,\dots ,v_n \}$ and $V(K_{2^{k-2}}) = \{1,\dots ,2^{k-2} \}.$ For all $i \in [n]$ let $H_i=G[\{ (v_i,v_j,p),(v_{i+1},v_j,p) : \, j \in [n], p \in [2^{k-2}] \}]$, where $i$ is taken with respect to modulo $n$. Finally let $S$  be a minimum percolating set of $G$.

Note that every vertex $(v_i,v_j,p) \in V(H_i)$ has exactly $3\cdot 2^{k-2}$ neighbors in $G-V(H_i)$. More precisely, $(v_i,v_j,p)$ is adjacent to vertices $(v_{i-1},v_{j'},p')$, where $j' \in \{j-1,j,j+1 \}$ and $p' \in [2^{k-2}]$. We infer that $|S \cap H_i| \geq 1,$ for all $i\in [n]$, which in turn implies that $|S| \geq \lceil \frac{n}{2} \rceil$.
\qed
\end{proof}

\subsection{At least three non-edge factors}
We can further improve Theorem~\ref{thm:veliko_faktorjev_vsaj_dva_3} if there are at least three non-edge factors. However, in this case we only obtain an upper bound.

\begin{thm}\label{thm:veliko_faktorjev_vsaj_trije_3}
Let $G$ be the strong product $G_1 \boxtimes \dots \boxtimes G_k$ of connected graphs $G_i$, $i\in [k]$. If at least three of the factors have order at least $3$, then $m(G,r)\leq 7\cdot 2^{k-3}$ for all $r\leq 7\cdot 2^{k-3}$.
\end{thm}

\begin{proof}
Since $m(G,r) \leq c$ for some $r$ and $c$ implies $m(G,r')\leq c$ for all $r' \leq r$, it suffices to show that $m(G,7\cdot 2^{k-3})= 7\cdot 2^{k-3}$. 

Let $G = G_1 \boxtimes \cdots \boxtimes G_k$ and denote $V(G_i) = \lbrace v_1^{(i)} , v_2^{(i)},\dots ,v_{n_i}^{(i)} \rbrace$ for all $i \in [k]$, and $n_1\ge n_2\ge n_3\ge 3$. Consider the paths $P^{(i)}:v_1^{(i)}v_2^{(i)}v_3^{(i)}$ in $G_i$ for all $i \in \{1,2\}$. Since $G_1$ and $G_2$ are connected of order at least $3$, such paths exist (note that $P^{(i)}$ is not necessarily induced). First, let $k=3$, and 
\begin{equation}
\label{eq:S'}
S'= V(P^{(1)}) \times V(P^{(2)}) \times  \{v_1^{(3)} \} \setminus \{(v_2^{(1)},v_2^{(2)},v_1^{(3)}), (v_1^{(1)},v_2^{(2)},v_1^{(3)})\}.
\end{equation}
Clearly, $|S'|=7$. Thus consider the spreading of infection in the $7$-bootstrap percolation. The first few infection steps are presented in the following table. (In the table, we use a simplified notation for the vertices. Notably, vertex $(v_i^{(1)},v_j^{(2)},v_{\ell}^{(3)})$ is written as $(i,j,\ell)$.) 
\begin{center}\label{Tabela1}
\begin{tabular}{ | m{2.3cm} | m{3.4cm}| m{8,5cm} | }
 \hline
 Infection step & Newly infected vertices & Infected by neighbors \\ 
 \hline
 Step 1 & $(2,2,1),(2,2,2)$ & $S'$ \\ 
 \hline
 Step 2 & $(3,2,2)$ & Step 1 and $(2,1,1),(3,1,1),(3,2,1),(3,3,1),(2,3,1)$ \\ 
 \hline
 Step 3 & $(2,1,2)$ and $(2,3,2)$ & Steps~$1$, $2$, and~$(1,1,1),(2,1,1),(3,1,1),(3,2,1)$ and $(1,3,1),(2,3,1),(3,3,1),(3,2,1)$, respectively.\\ 
 \hline
 Step 4.1 & $(1,2,1),(1,2,2)$ & Steps $1, 3$, and  $(1,1,1),(2,1,1),(1,3,1),(2,3,1)$ \\ 
 \hline
  Step 4.2 & $(3,1,2),(3,3,2)$ & Steps $1, 2$, and~$(2,1,1),(3,1,1),(3,2,1),(2,1,2)$ and  $(2,3,1),(3,3,1),(2,3,2),(3,2,1)$, respectively. \\ 
 \hline
 Step 5 & $(1,1,2)$ and $(1,3,2)$& Steps $1$, $4.1$, and $(2,1,2),(1,1,1),(2,1,1)$ and $(2,3,1)$, $(1,3,1),(3,1,1)$, respectively. \\ 
 \hline
\end{tabular}
\end{center}

After these infection steps, the vertices in $V(P^{(1)}) \times V(P^{(2)}) \times \{v_1^{(3)},v_2^{(3)}\}$ are all infected. Next, by replacing $v_1^{(3)}$ and $v_2^{(3)}$ in the previous steps with arbitrary adjacent vertices $v_s^{(3)}$ and $v_t^{(3)}$ in $G_3$ we deduce the following claim: $V(P^{(1)}) \times V(P^{(2)}) \times \{v_s^{(3)}\}$ being infected implies that vertices in $V(P^{(1)}) \times V(P^{(2)}) \times \{v_t^{(3)}\}$ also become infected. Since $G_3$ is connected this proves that $S'$ eventually infects all vertices in $V(P^{(1)}) \times V(P^{(2)}) \times V(G_3)$. 

Let $P'$ be an arbitrary path on three vertices in $G_3$ ($P'$ is well defined because $|V(G_3)| \geq 3$). By using analogous arguments as in the previous paragraph (where the roles of the first and the third coordinate are reversed) we derive that infected vertices $\{v_1^{(1)}\} \times V(P^{(2)}) \times V(P')$ also infect $V(G_1) \times V(P^{(2)}) \times V(P')$. Since $P'$ is arbitrary and every vertex in $V(G_3)$ lies on some path $P_3$, this implies that $V(G_1) \times V(P^{(2)}) \times V(G_3)$ becomes infected. 

Finally let $P''$ be an arbitrary path on three vertices  in $G_1$. Then infected vertices $V(P'') \times \{ v_1^{(2)} \} \times V(P')$ eventually infect $V(P'') \times V(G_2) \times V(P')$. Once again since $P'$ and $P''$ can be chosen arbitrarily, this implies that $V(G_1) \times V(G_2) \times V(G_3) = V(G)$ becomes infected.

Now let $k\geq 4$. For each $i\in \{4,\ldots, k\}$, let $F_i=\{v_1^{(i)}, v_2^{(i)}\}$ consist of two adjacent vertices. Let
$$S=S'\times \prod_{i=4}^k F_i,$$
where $S'\subset V(G_1)\times V(G_2)\times V(G_3)$ as defined in~\eqref{eq:S'}. 

Let $(x_1,x_2,x_3) \in V(G_1\boxtimes G_2 \boxtimes G_3)$ be a vertex infected in Step $1$ from the above table. Then, any vertex $(x_1,x_2,x_3,u_4,\ldots,u_k)$, where $u_i\in F_i$ for all $i\ge 4$, is adjacent to all vertices in $S$, therefore it has $7\cdot 2^{k-3}$ neighbors, and becomes infected. By using analogous constructions for Steps 2-5, we deduce that eventually every vertex in $V(P^{(1)}) \times V(P^{(2)}) \times \{v_1^{(3)},v_2^{(3)}\} \times \prod _{i=4}^k F_i$ becomes infected. This in turn implies that vertices in $V(P^{(1)}) \times V(P^{(2)}) \times V(G_3)\times \prod _{i=4}^k F_i$ become infected, and in the similar way as above, we then infer that  vertices in $V(G_1) \times V(G_2) \times V(G_3)\times \prod _{i=4}^k F_i$ also become infected. 

By continuing this process (choosing paths on three vertices in two coordinates among the first three coordinates, and exchanging the coordinate in which the infection is spread), we finally derive (in a similar way as in the proof of Theorem~\ref{thm:veliko_faktorjev_vsaj_dva_3}) that all vertices of $G$ eventually become infected, and thus $S$ percolates.   
\qed
\end{proof}

\subsection{No $K_2$ factors}

In this subsection we deal with the last line of the table in Section~\ref{sec:org}. 

\begin{thm}\label{thm:veliko_faktorjev_vsi_3}
Let $k \geq 4$ and $G$ be the strong product $G_1 \boxtimes \dots \boxtimes G_k$ of connected graphs $G_i$, $i\in [k]$. If $|V(G_i)| \geq 3$ for all $i \in [k]$, then $m(G,r)\leq 3^{k-1} -k$ for all $r\leq 2^k-1$.
\end{thm}

\begin{proof}
Note that it suffices to show that $m(G,2^k-1)\leq 3^{k-1}-k$. Indeed, the truth of this inequality directly implies that $m(G,r)\leq 3^{k-1}-k$ for all $r\le 2^{k}-1$. 

Let $V(G_i) = \lbrace v_1^{(i)} , v_2^{(i)},\dots ,v_{n_i}^{(i)} \rbrace$ for all $i \in [k]$, and by the assumption $n_i\ge 3$ for all $i\in [k]$. 
Let $P^{(i)}:v_1^{(i)}v_2^{(i)}v_3^{(i)}$ be paths in $G_i$ for all $i \in [k]$. (Since for each $i \in [k]$ graph $G_i$ is connected of order at least $3$, such a path exists.) 
Let $v=(v_2^{(1)},\ldots,v_2^{(k-1)},v_1^{(k)})$,  
$$U = \{(x_1,\ldots,x_{k-1},v_1^{(k)}): \, \exists i\in [k-1] \text{ with } x_i = v_3^{(i)},  \text{ and } x_j = v_2^{(j)} \, \forall j \neq i \} \bigcup \{v\},$$ and let
\begin{equation}
\label{eq:2}
S = \left(\prod_{i=1}^{k-1}V(P^{(i)})\times \{v_1^{(k)}\} \right)  \setminus U.
\end{equation}
Clearly, $|S|=3^{k-1}-k$. Since $k \geq 4$, it follows that $3^{k-1}-k \geq 2^k -1$. We claim that $S$ percolates. First note that $v$ is adjacent to every vertex in $S$, so it gets infected. Let $u \in U\setminus \{v\}$ be an arbitrary vertex. Without loss of generality let $u = (v_3^{(1)},v_2^{(2)},\ldots,v_2^{(k-1)},v_1^{(k)})$. Then $u$ is adjacent to $v$ and also to every vertex $(y_1,\ldots,y_{k-1},,v_1^{(k)})$, where $y_1 \in \{v_2^{(1)},v_3^{(1)}\}$ and $y_i \in P^{(i)}$ for all $ i \in \{2,\ldots,k-1\}$, except for the vertices in $U \setminus \{v\}$. Therefore it has $2\cdot 3^{k-2} -(k-1)$ infected neighbors, and since $k \geq 4$, we have $2\cdot 3^{k-2} -(k-1) \geq 2^k-1$, as desired. With this we have proved that vertices in $U$ get infected. Therefore, all vertices in $\prod_{i=1}^{k-1}V(P^{(i)})\times \{v_1^{(k)}\}$ are now infected. 

Next, we will prove that vertices in $\prod_{i=1}^{k-1}V(P^{(i)})\times V(G_k)$ become infected. Since $G_k$ is connected it suffices to show that if vertices in $\prod_{i=1}^{k-1}V(P^{(i)})\times \{v_s^{(k)}\}$ are infected this implies that vertices $\prod_{i=1}^{k-1}V(P^{(i)})\times \{v_t^{(k)}\}$ become infected, where $v_s^{(k)}$ and $v_t^{(k)}$ are any adjacent vertices in $G_k$.
Assume now that $v_s^{(k)}$ and $v_t^{(k)}$ are adjacent in $G_k$ and that vertices of 
$\prod_{i=1}^{k-1}V(P^{(i)})\times \{v_s^{(k)}\}$ are infected. Note that all these vertices are adjacent to the vertex $(v_2^{(1)},\ldots ,v_2^{(k-1)},v_t^{(k)})$, which gets infected, since it has $3^{k-1}$ infected neighbors. Let $j \in \{0,1,\ldots ,k-2\}$ and suppose that all vertices in $\prod_{i=1}^{k-1}V(P^{(i)})\times \{v_t^{(k)}\}$ which contain at most $j$ coordinates not equal to $v_2^{(i)}$ for some $i \in [k]$ are already infected. Consider any vertex $x \in \prod_{i=1}^{k-1}V(P^{(i)})\times \{v_t^{(k)}\}$ which contains exactly $j+1$ coordinates not equal to $v_2^{(i)}$ for some $i \in [k]$. Without loss of generality let $x = (v_3^{(1)},\ldots,v_3^{(j+1)},v_2^{(j+2)},\ldots,v_2^{(k-1)},v_t^{(k)}).$ Then $x$ has $2^{j+1}\cdot3^{k-j-2}$ infected neighbors of the form of $y=(y_1,\dots,y_{k-1},v_s^{(k)})$ where $y_i \in \{v_2^{(i)},v_3^{(i)}\}$ for all $i \in [j+1]$ and $y_i \in \{v_1^{(i)},v_2^{(i)},v_3^{(i)}\}$ for all $i \in \{j+2,\ldots,k-1\}.$ Note that $x$ also has $2^{j+1}-1$ infected neighbors of the form of $y=(y_1,\ldots,y_{j+1},v_2^{(j+2)},\ldots ,v_2^{(k-1)},v_t^{(k)})$, where $y_i \in \{v_2^{(i)},v_3^{(i)}\}$ for all $i \in [j+1]$, except when $y_i= v_3^{(i)}$ for every $i \in [j+1]$. Thus $x$ has $2^{j+1}\cdot3^{k-j-2} + 2^{j+1}-1 \geq 2^k -1$ infected neighbors, therefore it gets infected. This proves that eventually every vertex in $\prod_{i=1}^{k-1}V(P^{(i)})\times V(G_k)$ gets infected.

By continuing this process (choosing paths on three vertices in all but one coordinate, and spreading the infection throughout the remaining coordinate), we finally derive that all vertices of $G$ eventually become infected, and thus $S$ percolates.   
\qed
\end{proof}

Note that we only proved an upper bound for $m(G,r)$, where $G$ has $k$ factors all of which with at least three vertices, and the exact values are still open. Nevertheless, Observation~\ref{obs1} yields a dichotomy to the above result by presenting strong products of $k$ graphs in which the $r$-percolation number, where $r\ge 2^k$, can be arbitrarily large.

\section{Percolation numbers of $G\boxtimes H$}
\label{sec:two-factors}

In this section, we consider $m(G\boxtimes H,r)$, where $G$ and $H$ are non-trivial connected graphs. When $r=3$ and both $G$ and $H$ have order at least $3$, we immediately get the following result from Theorem~\ref{thm:veliko_faktorjev_vsaj_dva_3}. 
 
\begin{cor}
\label{prp:r3k2}
If $G$ and $H$ are connected graphs each with at least $3$ vertices, then $m(G \boxtimes H,3)=3$. 
\end{cor}

In the next subsection, we consider the only remaining case for $m(G\boxtimes H,3)$, which is when one of the factors is $K_2$, and prove a characterization of the graphs $G$ such that $m(G \boxtimes K_2,3)=3$.  In Section~\ref{sec:r4r5} we follow with some bounds on $m(G \boxtimes H,4)$ and $m(G \boxtimes H,5)$.

\subsection{Strong prisms and $r=3$}
By Corollary~\ref{prp:r3k2}, the only remaining case for the $3$-neighbor bootstrap percolation of the strong product of two factors is when one of the factors is $K_2$, that is, $G\boxtimes K_2$, or the so-called strong prism of a graph $G$.  We will denote the vertices of a strong prism as follows: letting $V(K_2)=[2]$, we will write $V(G\boxtimes K_2) = \lbrace v_i : \, v_i=(v,i), \textrm{ where } v \in V(G), i \in [2] \rbrace$.

Note that if $x$ and $y$ are twins in a graph $G$, and among the two only $x$ belongs to a percolating set of $G$, then $S'=(S\setminus\{x\})\cup\{y\}$ is also a percolating set of $G$, and has the same cardinality as $S$. This observation will be helpful in the proof of the following auxiliary result. 

\begin{lema} \label{lem:prizme}
If $G$ is a graph of order at least $3$ and $m(G \boxtimes K_2, 3)=3$, then there exists a minimum percolating set of $G \boxtimes K_2$ all vertices of which are in the same $G$-layer. 
\end{lema}

\begin{proof}
Let $S$ be a minimum percolating set of $G\boxtimes K_2$.
Note that for any $v\in V(G)$, vertices $v_1$ and $v_2$ are (closed) twins. Hence, if $S$ contains three vertices no two of which are in the same $K_2$-layer, then by the observation preceding the lemma, we infer that $S$ can be modified in such a way that all of its vertices belong to the same $G$-layer.

Now, assume that $S= \lbrace w_1,w_2,u_1 \rbrace$ and let $x \in V(G \boxtimes K_2)$ be a common neighbor of vertices $w_1,w_2,u_1$. Consider the following two cases. 
\begin{enumerate}
\item[(i)] If $x=u_2$, then $u_1$ and $w_1$ are neighbors. Since $|V(G)|\ge 3$, there exists a vertex $z_1$ adjacent to $u_1$ or $w_1$ (hence, $u, z$ and $w$ form a path in $G$ of which central vertex is $u$ or $w$). Note that $S'=\{u_1,z_1,w_1\}$ is a percolating set of $G\boxtimes K_2$, since after at most two steps $w_2$ gets infected as well.  
\item[(ii)] If $x \neq u_2$, then $x \in \lbrace z_1,z_2 \rbrace$ for a vertex $z \in V(G)\setminus \lbrace u,w \rbrace$. In either way, $z_1$ is adjacent to both $u_1$ and $w_1$, and so $uzw$ is a path in $G$. Again,  $S'=\{u_1,z_1,w_1\}$ is a percolating set of $G\boxtimes K_2$, since $z_2$ gets infected after the first step, and then $w_2$ is infected after the second step.
\end{enumerate}
In both cases, we found a minimum percolating set of $G \boxtimes K_2$ that lies in one $G$-layer, as desired. 
\qed
\end{proof}

Next we present a complete characterization of strong prisms whose $3$-percolation number equals $3$. 

\begin{thm} \label{thm:karakterizacijaprizem}
If $G$ is a connected graph, then $m(G\boxtimes K_2,3)=3$ if and only if either $m(G,2) = 2$ or $m(G,2)=3$ with a percolating set $S$ such that vertices from $S$ lie in a subgraph of $G$ isomorphic to $P_3$ or $K_{1,3}$. 
\end{thm}

\begin{proof}
Firstly, let $m(G,2)=2$ and let $\lbrace u,v \rbrace$ be a percolating set in the $2$-neighbor bootstrap percolation in $G$. If $G\cong K_2$, then $S=\lbrace u_1,v_1,u_2 \rbrace$ is a percolating set in the 3-neighbor bootstrap percolation in $G \boxtimes K_2$, so we may assume that $|V(G)|\ge 3$. Hence, there is a vertex $w$ in $G$ adjacent to both $u$ and $v$. Let $S = \lbrace u_1,v_1,w_1 \rbrace$, and consider the $3$-neighbor bootstrap percolation in $G\boxtimes K_2$. Note that $w_2$ gets infected directly from $S$, and in the second step also $u_2$ and $v_2$ get infected. From this point forward, the $3$-neighbor bootstrap percolation process in $G\boxtimes K_2$ follows analogous lines as the $2$-neighbor bootstrap percolation process in $G$. Notably, if $z\in V(G)$ gets infected by $x$ and $y$ in $G$, then $z_1\in V(G)\times [2]$ and $z_2\in V(G)\times [2]$ get infected by $\{x_1,x_2,y_1,y_2\}$ in $G\boxtimes K_2$. Thus $S$ percolates, and $m(G\boxtimes K_2,3)=3$. 

Secondly, let $m(G,2)=3$ and let $S=\lbrace u,v,w \rbrace$ be a percolating set of $G$ such that $uvw$ is a path in $G$. Let $S' = \lbrace u_1,v_1,w_1 \rbrace$, and consider the $3$-neighbor bootstrap percolation in $G\boxtimes K_2$. In the same way as in the previous paragraph we note that also $u_2,v_2$ and $w_2$ get infected. In addition, we  note that from this point forward, the $3$-neighbor bootstrap percolation process in $G\boxtimes K_2$ follows analogous lines as the $3$-neighbor bootstrap percolation process in $G$. Notably, if $w\in V(G)$ gets infected by $x, y$ and $z$ in $G$, then $w_1\in V(G)\times [2]$ and $w_2\in V(G)\times [2]$ get infected by $\{x_1,x_2,y_1,y_2,z_1,z_2\}$ in $G\boxtimes K_2$ (in fact, one could use only three vertices among the six to get the same result). 
Finally, let $S=\lbrace u,v,w \rbrace$ be a percolating set of $G$ such that $u,v,w$ are leaves in a subgraph of $G$ isomorphic to $K_{1,3}$ whose central vertex is denoted by $a$. Note that $S'=\lbrace u_1,v_1,z_1 \rbrace$ immediately infects $a_1$. Since vertices of the path $u_1a_1v_1$ in $G\boxtimes K_2$ are infected, we infer by the same reasoning as earlier that $S'$ percolates in $G \boxtimes K_2$.

For the reverse implication, let $m(G \boxtimes K_2,3)=3$. If $|V(G)|\le 3$ one can readily check that $m(G,2)=2$, so let $G$ be of order at least $4$. By Lemma \ref{lem:prizme}, one can choose a percolating set $S$ such that all its vertices lie in the same $G$-layer. Thus we may assume that $S=\lbrace u_1,v_1,w_1 \rbrace$ is a percolating set of $G\boxtimes K_2$, where $u,v,w \in V(G)$. Hence, there exists a common neighbor $x \in V(G\boxtimes K_2)$ of $u_1,v_1,w_1$. Consider the following cases:
\begin{enumerate}
\item[(i)] $x \in \lbrace u_2,v_2,w_2 \rbrace$. Then $u,v,w$ lie on a path $P_3$ in $G$. 
\item[(ii)] $x \notin \lbrace u_2,v_2,w_2 \rbrace$. Then $x \in \lbrace z_1,z_2 \rbrace$ for some vertex $z \in V(G)\setminus \lbrace u,v,w \rbrace$. Note that both $z_1$ and $z_2$ are adjacent to $u_1,v_1,w_1$. Hence, $u,v$ and $w$ are leaves of a subgraph of $G$ isomorphic to $K_{1,3}$ and $z$ is its central vertex.
\end{enumerate}
In both cases, either immediately or after the first step, one finds an infected path isomorphic to $P_3$ in the first layer, and as noted earlier, the corresponding vertices in the second layer also get infected. This property is maintained throughout the process; namely, whenever a vertex $u_1$ gets infected, its twin $u_2$ has the same set of (infected) neighbors and so it also get infected at the same time (and vice versa). 
Since $u_1$ had at least three infected neighbors, at least two of them were in $G^1$. We infer that the set $\lbrace u,v,w \rbrace$ percolates in the $2$-neighbor bootstrap percolation in $G$, therefore $m(G,2)\le 3$. In addition, if $m(G,2)=3$, then by the above, $\{u,v,w\}$ is a percolating set, where $u,v,w$ are in a subgraph isomorphic to $P_3$ or $K_{1,3}$. The proof is complete. \qed
\end{proof}

Unfortunately, we do not know of a structural characterization of the class of graphs with $m(G,2)=2$ or $m(G,2)=3$. One might wonder if the latter class of graphs always contains a $2$-neighbor bootstrap percolating set $S$ that appears in the formulation of the theorem, that is, vertices from $S$ lying in a subgraph of $G$ isomorphic to $P_3$ or $K_{1,3}$.  However, the following example shows this is not the case. Let $G'$ be the graph obtained from $C_4$ by adding two leaves to a vertex; see Fig.~\ref{fig:G'}. Note that $m(G',2)=3$ with the unique percolating set $S$ depicted in the figure (the leaves must be included in $S$ and one more vertex is needed), which does not satisfy the condition of Theorem \ref{thm:karakterizacijaprizem}, therefore $m(G' \boxtimes K_2,3)>3$. 

\begin{figure}[!ht]
\centering
\begin{tikzpicture}[scale=1, style=thick]
\def\vr{3pt}
\def\len{1}

\coordinate(a) at (0,0);
\coordinate(b) at (0,2);
\coordinate(c) at (1,1);
\coordinate(d) at (2,0);
\coordinate(e) at (2,2);
\coordinate(f) at (1,3);

\draw (a) -- (c);
\draw (b) -- (c);
\draw (c) -- (d);
\draw (c) -- (e);
\draw (b) -- (f);
\draw (e) -- (f);

\draw(a)[fill=black] circle(\vr);
\draw(b)[fill=white] circle(\vr);
\draw(c)[fill=white] circle(\vr);
\draw(d)[fill=black] circle(\vr);
\draw(e)[fill=white] circle(\vr);
\draw(f)[fill=black] circle(\vr);

\draw[anchor = east] (c) node {$x$};

\end{tikzpicture}
\caption{Graph $G'$ with a unique percolating set $S$}
\label{fig:G'}
\end{figure}
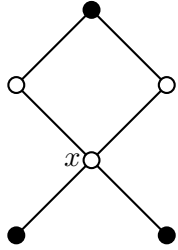

To gain a better understanding of the class of graphs characterized in Theorem~\ref{thm:karakterizacijaprizem}, we give some structural properties of the graphs $G$ with $m(G\boxtimes K_2,3)=3$ related to cut-vertices in $G$. 

\begin{prop} \label{prop: cut vertex}
If $m(G\boxtimes K_2,3)=3$, then $G$ has at most one cut-vertex $x$, and if $x$ is a cut-vertex of $G$ then $G-x$ has at most three components. 
\end{prop}

\begin{proof}
Let $G$ be a graph that either has two distinct cut-vertices, or $G$ has one cut-vertex $x$, such that $G-x$ has more than three components (both cases infer $|V(G)| \geq 4$). We will show that $m(G\boxtimes K_2,3)>3$. Suppose, to the contrary, that $m(G\boxtimes K_2,3)=3$. Then according to Theorem~\ref{thm:karakterizacijaprizem} this implies that either $m(G,2)=2$ or $m(G,2)=3$ with a percolating set $S$ such that vertices in $S$ lie on a subgraph of $G$ isomorphic to $P_3$ or $K_{1,3}$.  

Firstly let $x$ be an arbitrary cut-vertex of $G$ and let $S$ be a percolating set of $G$ of either size $2$ or $3$, under the $2$-neighbor bootstrap percolation. Let $K$ and $L$ be any two connected components of $G-x$. Since there are no edges between vertices in $K$ and vertices in $L$, we infer that $S$ contains at least one vertex from either $K$ and $L$.  Thus if $G-x$ has more than three connected components, $S$ would contain at least $4$ vertices, a contradiction.

Now, suppose that $G$ contains an additional cut-vertex $y$. If $G-\lbrace x,y \rbrace$ has three connected components, then from the same reason as above $S$ contains a vertex from each of these components. However, such a set $S$ does not satisfy the condition that its vertices lie on a subgraph isomorphic to $P_3$ or $K_{1,3}$.

Finally, if $G-\lbrace x,y \rbrace$ has only two connected components, then $x$ and $y$ are adjacent. Once again $S$ contains a vertex from each connected component of $G-\{x,y\}$. Now, if $|S|=2$, then there are no common neighbors between infected vertices and $S$ cannot percolate. If $|S|=3$, then vertices in $S$ cannot lie on a graph isomorphic to $P_3$ or $K_{1,3}$. \qed
\end{proof}

\medskip

Let $G = K_{1,3}$ and consider $m(G \boxtimes K_2,3)$. Note that $G$ has a cut-vertex $x$, which separates $G-x$ into three components. Also note that a set $S$ containing the leaves of $K_{1,3}$ infects $G$, therefore according to Theorem~\ref{thm:karakterizacijaprizem}, $m(G \boxtimes K_2,3) = 3$. Hence, having a cut-vertex that yields exactly three components is possible for $m(G\boxtimes K_2,3)=3$.  

To see that the inverse of Proposition~\ref{prop: cut vertex} does not hold, take the graph $G'$ from Fig.~\ref{fig:G'}. Note that $G'$ has just one cut-vertex $x$ and $G'-x$ has three components. However $m(G'\boxtimes K_2,3)> 3$ because $m(G',2)=3$ with a unique percolating set $S$, which does not satisfy the condition from Theorem~\ref{thm:karakterizacijaprizem}. Another example is the cycle $C_5$. It has no cut-vertices and it is easy to see that no vertex set of size two percolates in the $2$-neighbor bootstrap percolation process. This means that $m(C_5,2)=3$ since a set containing one vertex and both of its diametral vertices is a percolating set of size 3. This is in fact the only way to form a minimum percolating set, since the only other possible subset of three vertices is a path $P_3$, which does not infect any new vertices.  Since the condition of Theorem~\ref{thm:karakterizacijaprizem} is not satisfied, $m(C_5\boxtimes K_2,2)>3$. 

By Theorem~\ref{thm:k_faktorjev_rvelik}, we infer the following result, showing that $m(G \boxtimes K_2,3)$ can be arbitrarily large.

\begin{cor} \label{prop: Cn prizme}
Let $n \in \mathbb{N}$. Then $m(C_n \boxtimes K_2,3) = \lceil \frac{n}{2} \rceil +1$. 
\end{cor}

\subsection{Two factors and $r=4$ or $r=5$}
\label{sec:r4r5}

Next, we consider the $4$- and the $5$-neighbor bootstrap percolation in strong products of two factors. The following result is of similar flavor as Theorem~\ref{thm:karakterizacijaprizem} in the sense that we use the $2$-percolation numbers of factor graphs.

\begin{thm} \label{thm: implikacija za m(G krat H, 4)}
Let $G$ and $H$ be connected graphs such that $m(G,2)=2$ and $m(H,2)=2$. 
Then,  $m(G \boxtimes H,4) \leq 5$. In addition, if there exists a  percolating set of $G$ (or $H$) consisting of two adjacent vertices, then $m(G \boxtimes H,4) = 4$.

\end{thm}

\begin{proof}
Denote $V(G) = \lbrace g_1,\dots ,g_n \rbrace$ and $V(H) = \lbrace h_1, \dots , h_m \rbrace$ in such a way $\lbrace g_1,g_2 \rbrace$ and $\lbrace h_1,h_2 \rbrace$ are the percolating sets of $G$ and $H$, respectively. We also assume, renaming the vertices of $G$ and $H$ if necessary, that every vertex $g_i$, respectively $h_j$, is infected in $G$, respectively $H$, using the $2$-neighbor bootstrap percolation rule by some pair $g_{i_1},g_{i_2}$, respectively $h_{j_1},h_{j_2}$, where ${i_1}<{i_2} < i$ and ${j_1}<{j_2} < j$. 

First let $V(G) = \lbrace g_1,g_2 \rbrace$, in which case $g_1g_2 \in E(G)$. If $V(H) = \lbrace h_1,h_2 \rbrace$, then $m(G\boxtimes H,4)=4$. Hence let $|V(H)| \geq 3$ and let $S=\lbrace (g_1,h_1),(g_1,h_2),(g_2,h_1),(g_2,h_2) \rbrace$. Let $h_i$ be a vertex adjacent to $h_{i_1},h_{i_2}$ in $H$ and assume that all vertices $(g_1,h_j),(g_2,h_j)$ for all $j < i$ are already infected. Then $(g_1,h_i),(g_2,h_i)$ are both adjacent to vertices $(g_1,h_{i_1}),(g_2,h_{i_1}),(g_1,h_{i_2}),(g_2,h_{i_2})$. By induction, the whole $G \boxtimes H$ gets infected and $m(G\boxtimes H,4)=4$. 

Now let us assume that $|V(G)| \geq 3$ and $|V(H)| \geq 3$. Denote by $S_3$  the subgraph induced by $\{(g_i,h_j):\, i,j \in [3]\}$. Consider two cases for a percolating set $S$ in the $4$-neighbor bootstrap percolation.

\textbf{Case 1:} $g_1g_2 \notin E(G)$ and $h_1h_2 \notin E(H)$. Let $S=\lbrace (g_1,h_1),(g_1,h_2),(g_2,h_1),(g_2,h_2),$ $(g_3,h_1) \rbrace$ be a set of size $5$ and consider the following $4$-neighbor bootstrap  percolation process. Immediately $(g_3,h_3)$ is infected by (all) vertices in $S$. After that, the vertex $(g_2,h_3)$ is infected by $(g_3,h_3),(g_2,h_1),(g_2,h_2),(g_3,h_1)$.  Finally, vertices $(g_3,h_2)$ and $(g_1,h_3)$ are adjacent to $(g_1,h_2),(g_1,h_3),(g_2,h_2),(g_2,h_3),(g_3,h_3)$ and $(g_1,h_1),(g_1,h_2)$, $(g_3,h_1),(g_3,h_2),(g_3,h_3)$ respectively. This infects the whole $S_3$.

\textbf{Case 2:} $g_1g_2 \in E(G)$. Let $S=\lbrace (g_1,h_1),(g_1,h_2),(g_2,h_1),(g_2,h_2) \rbrace$ and consider the following $4$-neighbor bootstrap percolation process. Vertex $(g_3,h_1)$ is now adjacent to every vertex in $S$. Since we have now obtained the same set of infected vertices as in Case 1, this set infects $S_3$.

Now, let us assume that $S_3$ is already infected and consider the $4$-neighbor bootstrap percolation process. Let $g_i$ be a vertex infected in the $2$-neighbor bootstrap percolation process in $G$ by vertices $g_{i_1},g_{i_2}$ (where ${i_1},{i_2} < i$) and suppose that vertices $(g_k,h_1),\dots ,(g_k,h_p)$, for all $k < i$ and some $p \leq m$ are already infected in $G\boxtimes H$. Then $(g_i,h_p)$ gets infected by vertices $(g_{i_1},h_p),(g_{i_2},h_p),(g_{i_1},h_{p_1}),(g_{i_2},h_{p_1})$, $(g_{i_1},h_{p_2}),(g_{i_2},h_{p_2})$, where $h_p$ gets infected by $h_{p_1}$ and $h_{p_2}$ in the $2$-neighbor bootstrap percolation process in $H$. By the same argument we can show that every vertex $(g_i,h_j)$ where $3\leq j \leq p-1$ has $6$ infected neighbors and thus becomes infected. Finally $(g_i,h_2)$ and $(g_i,h_1)$ are adjacent to $(g_{i_1},h_2),(g_{i_2},h_2), (g_i,h_3), (g_{i_1},h_3),(g_{i_2},h_3)$ and $(g_{i_1},h_1),(g_{i_2},h_1), (g_i,h_3)$, $(g_{i_1},h_3), (g_{i_2},h_3)$ respectively. Thus all vertices in $\{g_1,\ldots,g_i\}\times\{h_1,\ldots,h_p\}$ are now infected. 

By reversing the roles of $G$ and $H$, we can deduce that all vertices in $\{g_1,\ldots,g_i\}\times\{h_1,\ldots,h_{p+1}\}$  become infected. By using induction, we infer that all vertices in $G\boxtimes H$ become infected.  Therefore $S$ percolates and $m(G\boxtimes H,4) \leq 5$, concluding the proof of the theorem.  
\qed
\end{proof}

Notice that throughout the infection process in the above proof, the only vertex which was infected by less than 5 vertices, was vertex $(g_2,h_3)$. Therefore it is not difficult to see that one can modify the proof of Theorem~\ref{thm: implikacija za m(G krat H, 4)} by adding the vertex $(g_2,h_3)$ to the set of initially infected vertices $S$, and obtain the following corollary. 

\begin{cor}\label{cor:implikacija za m(G krat H,5)}
Let $G$ and $H$ be connected graphs such that $m(G,2)=2$ and $m(H,2)=2$.
If $|V(G)| \geq 3, |V(H)| \geq 3$, then $m(G \boxtimes H,5) \leq 6$. In addition, if there exists a percolating set of $G$ (or $H$) consisting of two adjacent vertices, then $m(G \boxtimes H,5) = 5$.
\end{cor}

Consider $m(P_3 \boxtimes P_3,4)$, and denote $V(P_3) = \lbrace 1,2,3 \rbrace$. Suppose that $S$ is a percolating set of $P_3 \boxtimes P_3$ with $|S|=4$. Since vertices $(1,1),(1,3),(3,1)$ and $(3,3)$ are of degree $3$, they must all be in $S$. They are all adjacent to vertex $(2,2)$, which gets infected. After that every remaining vertex, namely $(1,2),(2,1),(2,3),(3,2)$, contains exactly $3$ infected neighbors, therefore $S$ does not percolate. Hence, $m(P_3 \boxtimes P_3,4)=5$.  
Now, consider $m(P_3 \boxtimes P_3,5)$ and suppose that $S$ is a percolating set with $|S|=5$. Then $(1,1),(1,3),(3,1)$ and $(3,3)$ are all in $S$. It is not difficult to check that by adding to $S$ any of the remaining vertices, $S$ does not percolate. These examples show that the upper bounds in Theorem~\ref{thm: implikacija za m(G krat H, 4)} and Corollary~\ref{cor:implikacija za m(G krat H,5)} are sharp. 

Let $H=H_1\boxtimes H_2$ for arbitrary connected graphs $H_1,H_2$. Since Theorem \ref{thm:produkt_k_faktorjev} states that $m(H,2)=2$, where a percolating set consists of two adjacent vertices, the following corollary of Theorem \ref{thm: implikacija za m(G krat H, 4)} is immediate.

\begin{cor}
If $m(G,2) = 2$ and $|V(G)| \geq 3$, then $m(G \boxtimes H_1 \boxtimes H_2,5) = 5$.
\end{cor}

\section{Strong products of infinite paths}
\label{sec:infinite}

In this section, we extend the concept of the $r$-neighbor bootstrap percolation to infinite graphs. In particular, we consider strong products of two-way infinite paths. In what follows, we extend the consideration of the $r$-neighbor bootstrap percolation also to the trivial case when $r=1$. Clearly, if $G$ is connected, then $m(G,1)=1$. 

Given an infinite graph $G$ and a positive integer $r$ we let $m(G,r)=\ell < \infty$ if $S$, where $|S|=\ell$, is a minimum set of vertices in $G$ that are initially set as infected, and an arbitrary vertex in $G$ becomes infected by the $r$-neighbor bootstrap percolation process in a finite number of steps. Otherwise, if there is no such finite set $S$, we let $m(G,r)= \infty$. 
By $${\rm fpt}(G) = \sup\{r: \, m(G,r) <\infty \}$$ we define the {\em finiteness percolation threshold} of a graph $G$. 

Clearly, if $G$ is the complete (infinite) graph, then ${\rm fpt}(G)=\infty$. In addition, ${\rm fpt}(G)=\infty$ is true for any finite graph $G$, since $m(G,r)\le |V(G)|$ holds for any finite graph and any positive integer $r$. On the other hand, it is easy to see that ${\rm fpt}(\mathbb{Z})=1$, where $\mathbb{Z}$ is the two-way infinite path. 
As usual, let $\mathbb{Z}^{\boxtimes,n} = \mathbb{Z} \boxtimes \cdots \boxtimes \mathbb{Z}$ stand for the strong product of $n$ two-way infinite paths. We simplify the notation $\mathbb{Z}^{\boxtimes,n}$ to $\mathbb{Z}^n$. We wish to determine ${\rm fpt}(\mathbb{Z}^n)$ for every $n \in \mathbb{N}$, and we can use some results from previous sections to bound this value. The following result follows immediately from Theorem~\ref{thm:veliko_faktorjev_vsi_3}.

\begin{cor}
\label{cor:Zn-lower}
For every $n \in \mathbb{N}$, ${\rm fpt}(\mathbb{Z}^n) \geq 2^n-1$. 
\end{cor}

The following upper bound for the threshold of $\mathbb{Z}^n$ comes with an easy proof. 

\begin{prop}\label{prop:zgornja_meja_threshold}
\label{prp:Zn-upper}
For every $n \in \mathbb{N}$, ${\rm fpt}(\mathbb{Z}^n) \leq 3^{n-1}$. 
\end{prop}

\begin{proof}
Let $G=\mathbb{Z}^n$ for some $n \in \mathbb{N}$ and assume that $m(G,3^{n-1}+1)< \infty$. Let $S$ be the percolating set of $G$. Without loss of generality and possibly by using a linear translation of $S$, we may assume that $S \subseteq [k]^n$ for some positive integer $k$. Let $x \in V(G) \setminus [k]^n$. Then $x$ has at most $3^{n-1}$ neighbors in $[k]^n$, therefore the infection cannot spread out from the box $[k]^n$, a contradiction. Thus, ${\rm fpt}(\mathbb{Z}^n) \leq 3^{n-1}$.  \qed
\end{proof}

Combining Corollary~\ref{cor:Zn-lower} with Proposition~\ref{prp:Zn-upper} we get the exact value of the finiteness percolation threshold of the strong grid $\mathbb{Z}^2$. 

\begin{cor}
${\rm fpt}(\mathbb{Z}^2)=3$. 
\end{cor}

The following result gives the finiteness percolation threshold of $\mathbb{Z}^3$, which needs more effort.

\begin{thm}
${\rm fpt}(\mathbb{Z}^3)=9$. 
\end{thm}

\begin{proof}
From Proposition~\ref{prop:zgornja_meja_threshold} we get that ${\rm fpt}(\mathbb{Z}^3)\leq 9$. To obtain the desired equality we  need to show that $m(\mathbb{Z}^3,9)<\infty$.
Let $S=[5]^3$ and let $G=\mathbb{Z}^3$. We will prove that $S$ percolates $G$ in the $9$-neighbor bootstrap percolation process. For this purpose we will first show that vertices in $[5]\times [5] \times [6]$ get infected. 

Firstly note that vertices in $\{2,3,4\}\times \{2,3,4\}\times \{6\}$ have $9$ neighbors in $S$. Namely, any vertex $(x_1,x_2,6)$ where $x_i \in \{2,3,4\}$ for $i \in [2]$ is adjacent to $(y_1,y_2,5)$, where $y_i \in \{x_i-1,x_i,x_i+1\}$ for $i \in [2]$. 
Secondly, vertex $x =(1,3,6)$ has $6$ neighbors in $S$, namely $(y_1,y_2,5)$ where $y_1 \in \{1,2\}$ and $y_2 \in \{2,3,4\}$. Vertex $x$ is also adjacent to vertices $(2,y_2,6)$ where $y_2 \in \{2,3,4\}$, which were infected the first step. Thus, $x$ has $9$ infected neighbors and gets infected. By using symmetric arguments, vertices $(3,1,6),(5,3,6)$ and $(3,5,6)$ also get infected. 
Thirdly, vertex $x=(1,2,6)$ is adjacent to $(y_1,y_2,5)$, where $y_1 \in \{1,2\}$ and $y_2 \in \{1,2,3\}$, which are all in $S$. Vertex $x$ is also adjacent to vertices $(2,y_2,6)$ for $y_2 \in \{2,3\}$ and also to $(1,3,6)$, all of which were infected in previous steps. Since $x$ has 9 infected neighbors, it gets infected. By symmetry, vertices $(1,4,6),(5,2,6),(5,4,6),(2,1,6)$, $(2,5,6),(4,1,6)$ and $(4,5,6)$ also get infected. 

So far, with the exception of corner vertices, that is, vertices $(1,1,6),(5,1,6),(2,5,6)$ and $(5,5,6)$, all other vertices in $[5]\times[5]\times \{6\}$ have been infected. Again by using symmetry we infer that with the exception of corner vertices all other vertices of $[5]\times[5]\times \{0,6\}$, $[5]\times\{0,6\}\times[5]$ and $\{0,6\}\times[5]\times[5]$ become infected. More precisely, the corner vertices that have not yet been infected are of the form $x=(x_1,x_2,x_3)$, where $x_i \in \{0,6\}$ for an $i \in [3]$ and $x_j \in \{1,5\}$ for all $j \neq i$. Now, without loss of generality, consider the corner vertex $x=(5,1,6)$. Note that $x$ is adjacent to vertices $(4,1,6),(4,2,6),(5,2,6),(4,1,5),(4,2,5),(5,1,5)$ and $(5,2,5)$ as well as to $(4,0,5)$ and $(6,2,5)$, all of which have been infected. By symmetric arguments we infer that the other three corner vertices of $[5]\times[5]\times [6]$ become infected, by which all vertices of $[5]\times[5]\times [6]$ are infected.

Now, by repeating this infection process, we can eventually infect any vertex in $[5]\times [5] \times \{-k,-k+1,\ldots , k-1,k\}$ for any $k \in \mathbb{Z}$. Finally, similarly as in the proof of Theorem~\ref{thm:veliko_faktorjev_vsi_3},   reversing the roles of factors,  we deduce that eventually every vertex in $\{-k,-k+1,\ldots , k-1,k\}^3$ becomes infected, therefore $S$ percolates. \qed
\end{proof}

\section{Concluding remarks}
\label{sec:conclude}
In this paper, we considered the $r$-neighbor bootstrap percolation of a strong product of $k$ graphs obtaining or bounding the values of $m(G_1\boxtimes \cdots\boxtimes G_k,r)$. The results are divided into several cases, in which $r$ can be expressed as a function of $k$. In the basic case, where $k\ge 2$ and $r\le 2^{k-1}$, we show $m(G_1\boxtimes \cdots\boxtimes G_k,r)=r$, by which we generalize the result~\cite[Theorem 3.1]{coe-2019} due to Coelho et al. This case is improved when there are at least two non-edge factors in the strong product. More precisely, when there are two non-edge factors and $r\le 3\cdot 2^{k-2}$, we get $m(G_1\boxtimes \cdots\boxtimes G_k,r)=r$, while for three non-edge factors and $r\le 7\cdot 2^{k-3}$, we get the upper bound $m(G_1\boxtimes \cdots\boxtimes G_k,r)\le 7\cdot 2^{k-3}$. The following question is thus natural.
 
\begin{q}
Let $G$ be the strong product $G_1\boxtimes \cdots\boxtimes G_k$ of connected graphs $G_1,\ldots, G_k$, among which there are $\ell$ non-edge factors. For which $\ell$, where $\ell\in \{3,\ldots, k\}$, it holds that  
$$m(G_1\boxtimes \cdots\boxtimes G_k,r)\le (2^{\ell}-1)\cdot 2^{k-\ell},$$
where $r\le (2^{\ell}-1)\cdot 2^{k-\ell}$\,? 
\end{q}
It is likely that in the answer to the question above, $\ell$ is expressed as a function of $k$. Clearly, when $\ell =3$ the answer is positive by Theorem~\ref{thm:veliko_faktorjev_vsaj_trije_3}.
In fact, when $\ell=3$ we suspect that the upper bound can be improved so that the equality 
$m(G_1\boxtimes \cdots\boxtimes G_k,r)=r$ holds for all $r\le 7\cdot 2^{k-3}$. We extend this into the following question.

\begin{q}
\label{ques2}
Let $G$ be the strong product $G_1\boxtimes \cdots\boxtimes G_k$ of connected graphs $G_1,\ldots, G_k$, among which there are $\ell$ non-edge factors. For which $\ell$, where $\ell\in \{3,\ldots, ,k\}$, it holds that  
$$m(G_1\boxtimes \cdots\boxtimes G_k,r)=r,$$
where $r\le (2^{\ell}-1)\cdot 2^{k-\ell}$\,? 
\end{q}
Note that the positive answer to Question~\ref{ques2} when $\ell=k$ would mean a considerable improvement of Theorem~\ref{thm:veliko_faktorjev_vsi_3}, which seems unlikely. 


In the case of two factors (that is, $k=2$) the situation has been resolved for $r\in\{2,3\}$, and in part also for $r\in\{4,5\}$. It would be interesting to find a generalization of Theorem~\ref{thm: implikacija za m(G krat H, 4)} in which $m(G\boxtimes H,r)$ would depend on the values of $m(G,r_1)$ and $m(H,r_2)$ for some $r_1$ and $r_2$, which are smaller than $r$. 

The main open problem arising in Section~\ref{sec:infinite} is to determine ${\rm fpt}(\mathbb{Z}^n)$ for all $n\geq 4$. In particular, the following question is also open:

\begin{q}
\label{ques3}
For which $n\ge 2$, we have ${\rm fpt}(\mathbb{Z}^n)=3^{n-1}$?
\end{q}
Clearly, the above question has a positive answer for $n\in\{2,3\}$. We suspect that as $n\to \infty$, ${\rm fpt}(\mathbb{Z}^n)<3^{n-1}$. 

It would also be interesting to consider the $r$-neighbor bootstrap percolation in other classes of infinite graphs.

\section*{Acknowledgement}
B.B. was supported by the Slovenian Research Agency (ARRS) under the grants P1-0297, J1-2452, J1-3002, and J1-4008.

\end{document}